\documentclass[12pt]{amsart}
\usepackage{amssymb,graphics}
\usepackage{hyperref}
\usepackage{tabularx}
\usepackage{booktabs}
\usepackage{caption}
\usepackage{mathdots}
\usepackage[usenames,dvipsnames]{xcolor}

\newtheorem{thm}{Theorem}[section]
\newtheorem{lem}[thm]{Lemma}

\newtheorem{prop}[thm]{Proposition}
\newtheorem{ex}[thm]{Example}

\newtheorem*{prob*}{Open problem}

\theoremstyle{definition}

\newtheorem{defi}[thm]{Definition}

\theoremstyle{remark}

\newtheorem*{rem*}{Remark}

\DeclareMathOperator{\id}{id}

\DeclareMathOperator{\Aut}{Aut}

\newcommand{\kringel}{\mathbin{\raise1pt\hbox{$\scriptstyle\circ$}}}
\newcommand{\pkt}{\mathbin{\raise0pt\hbox{$\scriptstyle\bullet$}}}

\newcommand{\C}{\mathbb{C}}

\newcommand{\Q}{\mathbb{Q}}

\newcommand{\tr}{\mathop{\rm tr}}
\newcommand{\ad}{{\rm ad}}

\newcommand{\Der}{{\rm Der}}

\newcommand{\diag}{\mathop{\rm diag}}

\newcommand{\La}{\mathfrak{a}}

\newcommand{\Lg}{\mathfrak{g}}
\newcommand{\Lh}{\mathfrak{h}}

\newcommand{\Ll}{\mathfrak{l}}
\newcommand{\Ln}{\mathfrak{n}}

\newcommand{\Lr}{\mathfrak{r}}
\newcommand{\Ls}{\mathfrak{s}}

\newcommand{\al}{\alpha}
\newcommand{\be}{\beta}
\newcommand{\ga}{\gamma}
\newcommand{\de}{\delta}
\newcommand{\ep}{\varepsilon}

\newcommand{\la}{\lambda}
\newcommand{\om}{\omega}

\newcommand{\ov}{\overline}

\newcommand{\ra}{\rightarrow}
\newcommand{\ck}{\checkmark}
\renewcommand{\phi}{\varphi}

\begin{document}


\title[]{Fixed-point-free automorphisms of solvable Lie algebras}
\author[D. Burde]{Dietrich Burde}
\author[K. Dekimpe]{Karel Dekimpe}
\address{Fakult\"at f\"ur Mathematik\\
Universit\"at Wien\\
Oskar-Morgenstern-Platz 1\\
1090 Wien \\
Austria}
\email{dietrich.burde@univie.ac.at}
\address{Katholieke Universiteit Leuven Campus Kulak Kortrijk\\
Etienne Sabbelaan 53\\
8500 Kortrijk\\
Belgium}
\email{karel.dekimpe@kuleuven.be}
\date{\today}


\begin{abstract}
In this paper, we investigate the existence of fixed-point-free automorphisms for finite-dimensional
Lie algebras. By a result of Jacobson, a Lie algebra admitting a fixed-point-free automorphism
is solvable. We prove that such a Lie algebra must be even strongly unimodular. We find a necessary and
sufficient criterion such that  a complex almost abelian Lie algebra admits a fixed-point-free automorphism.
For complex filiform Lie algebras we show that the existence of a fixed-point-free automorphism is equivalent
to not being characteristically nilpotent. 
\end{abstract}

\maketitle

\section{Introduction}

Fixed-point-free automorphisms of groups and Lie rings have been studied by many authors.
An automorphism $\phi$ of a group $G$ is called f.p.f.\ (fixed-point-free) if $C_G(\phi)=1$,
that is, the only fixed point of $\phi$ is $1$. For Lie rings, an automorphism is called f.p.f., if
$\phi(x)=x$ implies $x=0$ for all $x$.
Higman showed in $1957$ that a finite-dimensional Lie algebra over a field of characteristic zero
admitting an f.p.f.\ automorphism of prime order $p$ is nilpotent. It was then asked for a 
function $h(p)$, which is a good bound on the nilpotency class, depending on $p$. 
In $1963$ Kreknin showed then that if a Lie algebra $\Lg$ admits an an f.p.f.\ automorphism of
any finite order, then $\Lg$ is solvable. Jacobson proved the solvability also without 
the assumption that $\phi$ has finite order. It is also known that a finite group admitting an f.p.f.\ automorphism is solvable.
There are many other results, and a detailed survey can be found in the paper \cite{KHU} 
by Khukhro. \\[0.2cm]
In this paper we present some new results on Lie algebras admitting an f.p.f.\ automorphism. 
In section $3$, we show that any such Lie algebra is {\em strongly unimodular}. Here strongly unimodular Lie algebras
form a subclass of unimodular Lie algebras, arising in the context of lattices in Lie groups.
Garland showed in $1966$ that the Lie algebra of a (connected and simply connected) solvable Lie group admitting
a lattice is strongly unimodular. \\[0.2cm]
In section $4$ we classify all complex Lie algebras of dimension $n\le 4$ admitting an f.p.f.\ automorphism.
For such Lie algebras we also determine all possible finite orders of  f.p.f.\ automorphisms. In dimension $4$
we find a one-parameter family of Lie algebras, where all Lie algebras are strongly unimodular, but only
a few of them admit an f.p.f.\ automorphism. This family consists of almost abelian Lie algebras. \\[0.2cm]
In section $5$ we determine, motivated by the family of Lie algebras in dimension $4$, a necessary and sufficient
criterion for the existence of an f.p.f.\ automorphism for complex almost abelian Lie algebras.
It is related to $n$-cyclotomic operators, see Theorem $\ref{5.4}$. \\[0.2cm]
Finally, in section $6$ we prove that a complex filiform Lie algebra admits an f.p.f.\ automorphism if and only if it
is not characteristically nilpotent, if and only if it admits a nonsingular derivation, and if and only if it is a
derived Lie algebra, i.e., it equals $[\Lh,\Lh]$ for some Lie algebra $\Lh$. Here we are using results from a recent
paper \cite{SIT} by Siciliano and Towers.

\section{Preliminaries}

Let $\Lg$ be a finite-dimensional Lie algebra over a field $K$. We will assume that
$K$ has characteristic zero if not said otherwise.
Denote by $\Aut(\Lg)$ the group of Lie algebra automorphisms of $\Lg$.

\begin{defi}
A Lie algebra automorphism $\phi\in \Aut(\Lg)$ is called {\em fixed-point-free}, or f.p.f.,
if $\phi(x)=x$ implies $x=0$ for all $x\in \Lg$.
\end{defi}

Note that an automorphism $\phi\in \Aut(\Lg)$ is f.p.f., if and only if $\phi-\id$ is 
nonsingular. \\[0.2cm]
It is well known that the existence of an f.p.f.\ automorphism (of finite or infinite order) has
consequences for the algebraic structure of the Lie algebra $\Lg$. For a survey on some results see  
\cite{ZHA}. As we have mentioned in the introduction, Higman \cite{HIG} showed in $1957$, that the 
existence of an f.p.f.\ automorphism of prime order of $\Lg$ implies that $\Lg$ is nilpotent.

\begin{prop}[Higman]
Let $\Lg$ be a finite-dimensional Lie algebra over an arbitrary field admitting
an f.p.f. automorphism of order $p$ with $p$ prime. Then $\Lg$ is nilpotent.
\end{prop}

For $p=2$ we have the following result, which is much easier to prove, see \cite{ZHA}. 

\begin{prop}[Zha]\label{2.3}
Suppose that $\Lg$ admits an f.p.f.\ automorphism of order $2$. Then $\Lg$ is abelian.
\end{prop}

\begin{proof}
Let $\phi$ be an f.p.f.\ automorphism of order $2$. Then the linear transformation 
$\id-\phi$ is surjective, so that every element in $\Lg$ can be written as 
$x-\phi(x)$ for some $x$. Because of
\[
\phi(x-\phi(x))=\phi(x)-\phi^2(x)=-(x-\phi(x))
\]
it follows that $\phi(x)=-x$ for all $x\in \Lg$. This implies
\[
\phi([x,y])=[\phi(x),\phi(y)]=[-x,-y]=[x,y]
\]
for all $x,y\in \Lg$. Thus $[x,y]=0$ for all $x,y\in \Lg$.
\end{proof}

A much more general result is the following one by Kreknin \cite{KRE}.

\begin{prop}[Kreknin]
Let $\Lg$ be a finite-dimensional Lie algebra over an arbitrary field admitting
an f.p.f.\ automorphism of order $n$. Then $\Lg$ is solvable.
\end{prop}

This result is still true if $\Lg$ admits any f.p.f.\ automorphism, not necessarily of finite order.
This has been shown by Jacobson \cite{JAC}, see Theorem $9$.

\begin{prop}[Jacobson]
Let $\Lg$ be a finite-dimensional Lie algebra over a field of characteristic zero admitting
an f.p.f.\ automorphism. Then $\Lg$ is solvable.
\end{prop}

The converse of this proposition is not true.

\begin{ex}\label{2.6}
The $2$-dimensional solvable non-abelian Lie algebra $\Lr_2(K)$ does not admit an f.p.f.\ automorphism.
\end{ex}

Indeed, let $\{e_1,e_2\}$ be a basis, with $[e_1,e_2]=e_1$. Then any automorphism
has the form
\[
\phi=\begin{pmatrix} \al & \be \cr 0 & 1 \end{pmatrix}
\]
for some scalars $\al,\be \in K$. But we have $\det(\phi-\id)=0$, and $\phi$ has a nonzero fixed-point. \\[0.2cm]
Denote by $\Ln$ the nilradical of $\Lg$. Let $\Ln^k=[\Ln,\Ln^{k-1}]$ be the $k$-th term in the descending central
series of $\Ln$ for all $k\ge 1$, with $\Ln^1=\Ln$.

\begin{defi}\label{2.7}
A Lie algebra $\Lg$ is called {\em strongly unimodular}, if it is solvable and if for all $x\in \Lg$ the
induced operator of $\ad(x)$ on each subspace $\Ln^k/\Ln^{k+1}$ has trace zero, for all $k\ge 1$. \\
A Lie algebra $\Lg$ is called {\em unimodular}, if $\tr(\ad (x))=0$ for all $x\in \Lg$.
\end{defi}  

Strongly unimodular Lie algebras arise in the study of lattices in solvable Lie groups.
The following result was proved by Garland \cite{GAR}.

\begin{thm}[Garland]
Let $G$ be a connected and simply connected solvable Lie group admitting a cocompact lattice.
Then its Lie algebra is strongly unimodular.
\end{thm}

It is well known that a strongly unimodular Lie algebra is unimodular. For completeness we will
give a proof here.

\begin{lem}\label{2.9}
A strongly unimodular Lie algebra $\Lg$ is unimodular.
\end{lem}

\begin{proof}
As vector space, $\Lg$ corresponds to the direct sum 
\begin{align*}
\Lg/\Ln \oplus \Ln/\Ln^2\oplus \Ln^2/\Ln^3 \oplus \cdots \oplus \Ln^c/\Ln^{c+1}, 
\end{align*}
where $\Ln^{c+1}=0$. Since $\Lg$ is solvable, $[\Lg,\Lg]$ is nilpotent, so that $[\Lg,\Lg]\subseteq \Ln$.
Hence the quotient Lie algebra $\Lg/\Ln$ is abelian, and we have
\begin{align*}
\tr(\ad(x)) & =\tr(\ad(x)_{0})+\tr(\ad(x)_{1})+ \cdots + \tr(\ad(x)_{c}) \\
            & = 0
\end{align*}
for all $x\in \Lg$, where $\ad(x)_i$ ($1\leq i \leq c$) is the induced operator on the quotient $\Ln^{i}/\Ln^{i+1}$ and $\ad(x_0) =0$ (which is the induced operator on $\Lg/\Ln$).
\end{proof}  

The converse statement need not be true.

\begin{ex}
Let $\Lg$ be the solvable non-nilpotent Lie algebra with basis $\{e_1,\ldots ,e_5\}$ and
defining Lie brackets
\vspace*{0.5cm}
\begin{align*}
[e_1,e_5] & = 2\al e_1,  & [e_2,e_5] & = \al e_2+e_3,    & [e_4,e_5]& =-4\al e_4,\\
[e_2,e_3] & = e_1,       & [e_3,e_5] & = -e_2+\al e_3,   &                  \\
\end{align*}
where $\al\in \C^{\times}$. Then $\Lg$ is unimodular but not strongly unimodular,
for all $\al\neq 0$.
\end{ex}

Indeed, the trace of $\ad(e_5)$, induced on $\Ln/\Ln^2={\rm span}\{\bar e_2,\bar e_3,\bar e_4\}$ (with $\bar e_i = e_i + \Ln^2$) equals $2\al\neq 0$.

\section{Fixed-point-free automorphisms and strongly unimodular Lie algebras}

It turns out that there is a strong link between the existence of an f.p.f.\ automorphism 
for a Lie algebra and strong unimodularity. We will show the following result. 

\begin{thm}
Let $\Lg$ be a finite-dimensional Lie algebra over a field of characteristic zero admitting
an f.p.f.\ automorphism. Then $\Lg$ is strongly unimodular.
\end{thm}  

\begin{proof}
By Jacobson's theorem in \cite{JAC}, $\Lg$ is solvable. Let $\phi$ be an f.p.f.\ automorphism of $\Lg$.
We may assume that the field $K$ is algebraically closed, and that $\phi$ is semisimple. Indeed, the semisimple
part $\phi_s$ of $\phi$ is f.p.f. if and only if $\phi$ is f.p.f. We will prove the result by induction on
the nilpotency class $c$ of the nilradical $\Ln$ of $\Lg$. \\[0.2cm]
Let $c=1$, i.e., let $\Ln$ be abelian. Since $\phi$ is semisimple and $K$ is algebraically closed, we can
find a basis of eigenvectors $\{a_1,a_2, \ldots , a_m, b_1, b_2, \ldots, b_n\}$ of $\Lg$, such that
$\{ b_1, b_2, \ldots, b_n\}$ is a basis of $\Ln$. So there exist constants $\alpha_i, \beta_j \in K$ with
$\varphi(a_i) = \alpha_i a_i$ and $\varphi(b_j)= \beta_j b_j$.
As $\varphi$ is f.p.f., we have that $\alpha_i$ and $\beta_j$ are neither equal to 0 nor to 1. We have to show
that for each $1\leq i \leq m$, the linear map
\[                                                                                                  
\ad(a_i)_{|\Ln}: \Ln \to \Ln,\;  x \mapsto [a_i,x]                                                     
\]
has trace zero. Let $c^i_{kj}\in K$ be such that
\begin{equation}\label{brackets}                                                                    
[a_i,b_j]=c^i_{1j} b_1 + c^i_{2j} b_2 + \cdots + c^i_{kj} b_k + \cdots + c^i_{nj} b_n.          
\end{equation}
Then we have to show that the matrix $\left( c^i_{kj} \right)_{1\leq k,j \leq n}$ (which is the matrix representation
of $\ad(a_i)_{|\Ln}$) has trace zero for all $i$.
By applying $\varphi$ to both sides of \eqref{brackets}, we find that
\begin{eqnarray*}                                                                                   
\lefteqn{ \alpha_i \beta_j (c^i_{1j} b_1 + c^i_{2j} b_2 + \cdots + c^i_{kj} b_k + \cdots + c^i_{nj} b_n )}\\
                                                     & = &  \alpha_i \beta_j [a_i,b_j]  \\          
                                                     & = & \varphi( [a_i,b_j] ) \\                  
                                                     & = & \varphi (c^i_{1j} b_1 + c^i_{2j} b_2 +
 \cdots + c^i_{kj} b_k + \cdots + c^i_{nj} b_n)\\                                                 
                                                     & = & c^i_{1j} \beta_1 b_1 + c^i_{2j} \beta_2 b_2
                                                           + \cdots + c^i_{kj} \beta_k b_k + \cdots
                                                           + c^i_{nj} \beta_n b_n.                            
\end{eqnarray*}
By comparing the coefficient of $b_j$ in the top and bottom expression above, we find that
\[
c^i_{jj} \alpha_i \beta_j = c^i_{jj}  \beta_j.
\]
Using the fact that $\beta_j\neq 0$ and $\alpha_i \neq 1$, we find that $c^i_{jj}=0$. So the diagonal entries of
$\left( c^i_{kj} \right)_{1\leq k,j \le n}$ are all zero and hence this matrix has trace zero. \\[0.2cm]
We now assume that the nilpotency class of $\Ln$ is  $c>1$ and that the proposition is correct for 
smaller nilpotency classes. Note that $\varphi$ induces an f.p.f.\ automorphism on $\Lg/\Ln^c$ and hence by the
induction hypothesis $\Lg/\Ln^c$ is strongly unimodular. This shows that for all $x\in \Lg$, the map  $\ad(x)$ induces a
linear map with trace zero on all $\Ln^k/\Ln^{k+1}$ for $1 \leq k < c$. So we still have to show that $\ad(x)$ induces
a map with zero trace on $\Ln^c/\Ln^{c+1}= \Ln^c$. We can now fix a basis
$\{a_1,a_2, \ldots , a_m, c_1, c_2, \ldots c_l, b_1, b_2, \ldots, b_n\}$ of $\Lg$ consisting of eigenvectors of $\varphi$
such that $\{ c_1, c_2, \ldots c_l, b_1, b_2, \ldots, b_n\}$ is a basis of $\Ln$ and $\{ b_1, b_2, \ldots, b_n\}$
is a basis of $\Ln^c$. It now suffices to show that the maps
\[                                                                                                  
\ad(a_i)_{|\Ln^c}: \Ln^c \to \Ln^c: x \mapsto [a_i,x]                                               
\]
have trace zero. This follows exactly in the same way as above in the abelian case, which finishes
the proof.
\end{proof}  

The converse of this result does not hold. A strongly unimodular Lie algebra need not admit an
f.p.f.\ automorphism. For a family of such Lie algebras see Example $\ref{4.4}$ in the next section.

\section{A classification in low dimension}

In this section we will give a classification of all complex solvable Lie algebras of dimension $d\le 4$ admitting an
f.p.f. automorphism. It is clear that every complex abelian Lie algebra admits such an automorphism.
Indeed, $\phi=\diag (\zeta_n,\ldots ,\zeta_n)$ is a fixed-point free automorphism of order
$n$ for every $n\ge 2$. Here $\zeta_n$ denotes a primitive $n$-th root of unity. So we may assume that
$\Lg$ is non-abelian.

\begin{prop}\label{4.1}
Let $\Lg$ be a non-abelian complex Lie algebra of dimension $2$. Then $\Lg$ does not admit an f.p.f.\ automorphism.
\end{prop}

\begin{proof}
Since $\Lg$ is isomorphic to $\Lr_2(\C)$, the claim follows from Example $\ref{2.6}$.
\end{proof}  

In dimension $3$ we use the following classification list.
\vspace*{0.5cm}
\begin{center}
\begin{tabular}{c|c|c}
$\Lg$ & Lie brackets & strongly unimodular\\
\hline
$\C^3$ &  $-$ & $\ck$ \\
$\Ln_3(\C)$ & $[e_1,e_2]=e_3$ & $\ck$ \\
$\Lr_2(\C) \oplus  \C$ & $[e_1,e_2]=e_2$ & $-$ \\
$\Lr_3(\C)$ & $ [e_1,e_2]=e_2, [e_1,e_3]=e_2+e_3 $ & $-$ \\
$\Lr_{3,\la}(\C)$ & $[e_1,e_2]=e_2, [e_1,e_3]=\la e_3,\, \la\neq 0$ & $\ck \Leftrightarrow \la=-1$ \\
$\Ls\Ll_2(\C)$ & $[e_1,e_2]=e_3, [e_1,e_3]=-2 e_1,[e_2,e_3]=2 e_2$ & $-$ \\
\end{tabular}
\end{center}
\vspace*{0.5cm}
Note that $\Lr_{3,\la}(\C)\cong \Lr_{3,\mu}(\C)$ if and only if $\mu=\la^{-1}$, or $\mu=\la$.

\begin{prop}
Let $\Lg$ be a complex non-abelian Lie algebra of dimension $3$. Then
$\Lg$ admits an f.p.f.\ automorphism if and only if $\Lg$ is isomorphic to $\Ln_3(\C)$ or $\Lr_{3,-1}(\C)$.
All of these Lie algebras admit an f.p.f.\ automorphism of finite order, and
the possible orders are as follows.
\begin{itemize}
\item[(1)] $\Ln_3(\C)$ admits an f.p.f.\ automorphism of order $n$ for all $n\ge 3$.
\item[(2)] $\Lr_{3,-1}(\C)$ admits an f.p.f.\ automorphism of order $2m$ for all $m\ge 2$. 
\end{itemize}  
\end{prop}

\begin{proof}
The only $3$-dimensional strongly unimodular Lie algebras are $\C^3$, the Heisenberg Lie algebra
$\Ln_3(\C)$ and the solvable, non-nilpotent Lie algebra $\Lr_{3,-1}(\C)$. It is easy to see that these
Lie algebras admit an f.p.f.\ automorphism. In fact,
$\phi=\diag (\zeta_n,\zeta_n,\zeta_n^2)$ is an f.p.f.\ automorphism of order $n$ for every $n\ge 3$ for the
Heisenberg Lie algebra. Every f.p.f.\ automorphism of $\Lr_{3,-1}(\C)$ is of the form
\[
\phi=\begin{pmatrix} -1 & 0 & 0 \cr \al & 0 & \be \cr \ga & \de & 0 \end{pmatrix}
\]
with $\det(\phi)=\be\de\neq 0$ and $\det(\phi-\id)=2(\be\de-1)\neq 0$. It is obvious, that $\phi^{2m+1}$
cannot be the identity, because its entry at position $(1,1)$ equals $-1$, and not $1$.
On the other hand, for a primitive $m$-th root of unity $\zeta_m$,
\[
\phi=\begin{pmatrix} -1 & 0 & 0 \cr 0 & 0 & \zeta_m \cr 0 & 1 & 0 \end{pmatrix}
\]
is an automorphism of $\Lr_{3,-1}(\C)$ and has order $2m$ for all $m\ge 1$. It is f.p.f. if $1\cdot \zeta_m\neq 1$,
i.e., for all $m\ge 2$.
\end{proof}

In dimension $4$ we use the following classification from \cite{AGA}.
\vspace*{0.5cm}
\begin{center}
\begin{tabular}{c|c|c}
$\Lg$ & Lie brackets & s.u.\\
\hline
$\Lg_0=\C^4$ &  $-$  & $\ck$ \\
$\Lg_1=\Ln_3(\C)\oplus \C$ & $[e_1,e_2]=e_3$ & $\ck$ \\
$\Lg_2=\Ln_4(\C)$ & $[e_1,e_2]=e_3,\, [e_1,e_3]=e_4$ & $\ck$ \\
$\Lg_3=\Lr_2(\C) \oplus  \C^2$ & $[e_1,e_2]=e_2$ & $-$ \\
$\Lg_4=\Lr_2(\C) \oplus \Lr_2(\C)$ & $[e_1,e_2]=e_2, \,[e_3,e_4]=e_4$ & $-$ \\
$\Lg_5=\Ls\Ll_2(\C) \oplus \C$ & $[e_1,e_2]=e_2, \,[e_1,e_3]=-e_3,\,[e_2,e_3]=e_1$ & $-$ \\
$\Lg_6$ & $[e_1,e_2]=e_2, \,[e_1,e_3]=e_3,\,[e_1,e_4]=e_4$ & $-$ \\
$\Lg_7(\al)$ & $[e_1,e_2]=e_2, \,[e_1,e_3]=e_3,\,[e_1,e_4]=e_3+\al e_4$ & $\ck \Leftrightarrow \al=-2$\\
$\Lg_8$ & $[e_1,e_2]=e_2, \,[e_1,e_3]=e_3,\,[e_1,e_4]=2e_4,\,[e_2,e_3]=e_4$ & $-$ \\
$\Lg_9(\al,\be)$ & $[e_1,e_2]=e_2, \,[e_1,e_3]=e_2+\al e_3,\,[e_1,e_4]=e_3+\be e_4$ & $\ck \Leftrightarrow \al+\be=-1$\\
$\Lg_{10}(\al)$ & $[e_1,e_2]=e_2, \,[e_1,e_3]=e_2+\al e_3,\,[e_1,e_4]=(\al +1)e_4,$ & $\ck \Leftrightarrow \al=-1$\\
               & $[e_2,e_3]=e_4$ & \\
\end{tabular}
\end{center}
\vspace*{0.5cm}
Here $\Lg_{10}(\al)\cong \Lg_{10}(\al')$ if and only if $\al\al'=1$ or $\al=\al'$,
and $\Lg_9(\al_1,\be_1)\cong \Lg_9(\al_2,\be_2)$ if and only if the ratios
$1:\al_1:\be_1$ and $1:\al_2:\be_2$ coincide, after some permutation.
For $\al,\be\neq 0$ this means,
\[
\Lg_9(\al,\be) \cong \Lg_9(\al',\be')
\]
if and only if $(\al',\be')$ is equal to one of the following:
\[
(\al,\be),\; (\be,\al),\; 
\left( \frac{1}{\al}, \frac{\be}{\al} \right),\;
\left( \frac{\be}{\al}, \frac{1}{\al} \right),\; 
\left( \frac{1}{\be}, \frac{\al}{\be} \right),\;
\left( \frac{\al}{\be}, \frac{1}{\be} \right).
\]
Note that $\Lg_9(-1,0)\cong \Lr_{3,-1}(\C)\oplus \C$. 
The classification of $4$-dimensional Lie algebras admitting an f.p.f.\ automorphism is as follows.

\begin{prop}\label{4.3}
Let $\Lg$ be a non-abelian $4$-dimensional complex Lie algebra. Then
$\Lg$ admits an f.p.f.\ automorphism if and only if $\Lg$ is isomorphic to $\Ln_3(\C)\oplus \C$, $\Ln_4(\C)$,
$\Lr_{3,-1}(\C)\oplus \C$, $\Lg_9(\om,\om^2)$, or $\Lg_{10}(-1)$, where $\om$ is a primitive third root of unity.
All of these Lie algebras admit an f.p.f.\ automorphism of finite order, and
the possible orders are as follows.
\begin{itemize}
\item[(1)] $\Ln_3(\C)\oplus \C$ admits an f.p.f.\ automorphism of every order $n\ge 3$.
\item[(2)] $\Ln_4(\C)$ admits an f.p.f.\ automorphism  of every order $n\ge 4$.
\item[(3)] $\Lr_{3,-1}(\C)\oplus \C$ admits an f.p.f.\ automorphism of every order $2m$, with $m\ge 2$.
\item[(4)] $\Lg_9(\om,\om^2)$ admits an f.p.f.\ automorphism of every order $3m$, with $m\ge 2$.
\item[(5)] $\Lg_{10}(-1)$ admits an f.p.f.\ automorphism of every order order $2m$, with $m\ge 3$.
\end{itemize}
\end{prop}

\begin{proof}
The $4$-dimensional strongly unimodular Lie algebras are given by $\C^4$, $\Ln_3(\C)\oplus \C$,
$\Ln_4(\C)$, $\Lg_7(-2)$, $\Lg_9(\al,\be)$ with $\al+\be+1=0$, and $\Lg_{10}(-1)$. It is easy to see that
$\Lg_7(-2)$ does not admit an f.p.f.\ automorphism. In fact, a short computation shows that $1$ must be an eigenvalue.
The algebra $\Ln_3(\C)\oplus \C$ certainly admits an f.p.f.\ automorphism of every order $n\ge 3$, and cannot admit
one of order $2$, because it is non-abelian. For $\Ln_4(\C)$, the semisimple automorphisms are of the form
$\phi=\diag (s,t,st,s^2t)$. The possible orders of such $\phi$, which are f.p.f., are all $n\ge 4$. We may
take $s=t=\zeta_n$. \\[0.2cm]
The Lie algebras  $\Lg_9(\al,\be)$ with $\al+\be+1=0$ admit an f.p.f.\ automorphism if and only if $\al$ is a primitive
third root of unity and $\be=\al^2$, see also section $5$. 
For  $\Lg_9(\om,\om^2)$ one can verify by a direct computation, that the f.p.f.\ automorphisms $\phi$
have entry $(1,1)$ different from $1$ for powers $m>1$, which are not a multiple of $3$. On the other hand,
consider the following automorphisms, with complex parameter $\la$,
\[
\phi=\begin{pmatrix} \om & 0 & 0 & 0 \cr 0 & \la & 0 & 0 \cr 0 & (\om^2-1)\la & \om^2\la & 0 \cr
0 & 3\om\la & (\om-1)\la & \la\om
\end{pmatrix}
\]
We have $\det(\phi)=\om\la^3$ and $\det(\phi-\id)=(\om-1)(\la^3-1)$. Hence we need
$\la\neq 0,\la^3\neq 1$. We have $\phi^{3m}=\id$ if and only if $\la^{3m}=1$. So we can take $\la=\zeta_{3m}$, where we need
$m\ge 2$ because of $\la^3\neq 1$. \\[0.2cm]
Finally, For $\Lg_{10}(-1)$, every f.p.f.\ automorphism is of the form
\[
\phi=\begin{pmatrix} -1 & 0 & 0 & 0 \cr \al & \be & \ga & 0 \cr \de & -2\be & -\be & 0 \cr
\ep & -\be(2\al+\de) & \ga\de-\be(\al+\de) & \be(2\ga-\be)
\end{pmatrix}
\]
satisfying
\begin{align*}
\be(2\ga-\be) & \neq 0, \\
(2\be\ga-\be^2+1)(2\be\ga-\be^2-1) & \neq 0.
\end{align*}
Every odd power greater than $1$ of $\phi$ has coefficient $-1$ at the entry $(1,1)$. Hence $\phi$ cannot have odd order.
On the other hand,
\[
\phi=\begin{pmatrix} -1 & 0 & 0 & 0 \cr 0 & \zeta_m & 0 & 0 \cr 0 & -2\zeta_m & -\zeta_m & 0 \cr
0 & 0 & 0 & -\zeta_m^2
\end{pmatrix}
\]
is an f.p.f.\ automorphism of $\Lg_{10}(-1)$ of order $2m\ge 6$. For $m=2$, we have $\det(\phi-\id )=0$, which is
a contradiction.
\end{proof}

At the end of section $3$ we have mentioned that there are families of strongly nilpotent Lie algebras, which do
not admit an f.p.f.\ automorphism. We are able to present such a family now by considering the algebras
$\Lg_9(\al,\be)$.

\begin{ex}\label{4.4}
The Lie algebras $\Lg_9(\al,-1-\al)$ are strongly unimodular for all  $\al \in \C$, but do not
always admit an f.p.f.\ automorphism.   
\end{ex}  

Note that the Lie algebras $\Lg_9(\al,\be)$ are {\em almost abelian}. We will study f.p.f.\ automorphisms of such
Lie algebras in the next section.

\section{Automorphisms of almost abelian Lie algebras}

For this section, we assume that all algebras are defined over the field of complex numbers.
However, let us recall the definition of an almost abelian Lie algebra over an arbitrary field $K$.

\begin{defi}
A finite-dimensional Lie algebra $\Lg$ over a field $K$ is called {\em almost abelian}, if it is
non-abelian and has an abelian ideal of codimension $1$.
\end{defi}  

If $\Lg$ is almost abelian, with $\dim(\Lg)=n$, then $\Lg\cong \La \rtimes \langle v \rangle$,
where $\La$ is abelian of dimension $n-1$ and $v\neq 0$ in $\Lg$. An almost abelian Lie algebra
is $2$-step solvable. Its nilradical is given by $\Ln=\La$, so that $\Ln^2=0$ and
$\tr (\ad(x))=\tr (\ad(x)_{\mid \Ln})$. Hence an almost abelian Lie algebra is strongly unimodular if and only
if it is unimodular.

\begin{defi}\label{5.2}
Let $A\in M_m(\C)$ and $n\ge 2$ be an integer. We say that $A$ is {\em $n$-cyclotomic} if there exists a matrix
$B\in M_k(\C)$ with $m=kn$ such that $A$ is similar to the block matrix
\[
B_{\zeta_n}=  
\begin{pmatrix}
B & 0 & 0 & \cdots & 0 \cr
0 & \zeta_n B & 0 & \cdots & 0 \cr
0 & 0 & \zeta_n^2 B & \cdots & 0 \cr
\vdots & \vdots & \vdots & \ddots & \vdots \cr
0 & 0 & 0 & \cdots & \zeta_n^{n-1}B 
\end{pmatrix}  
\]  
\end{defi}  

We assume that $n\ge 2$, because otherwise every matrix $A\in M_m(\C)$ is $1$-cyclotomic. Note that every
$n$-cyclotomic matrix $A$ has trace zero, because we have
\[
\tr (A)=(1+\zeta_n+\zeta_n^2+\cdots +\zeta_n^{n-1})\tr (B)=0.
\]  
Furthermore $A$ is invertible if and only if $B$ is invertible, because we have
\[
\det(A)=(1\cdot \zeta_n\cdot \zeta_n^2\cdots \zeta_n^{n-1})^k\det(B)^n=\pm \det(B)^n.
\]  
Denote by $I_k$ the identity matrix of size $k$.

\begin{lem}\label{5.3}
Let $B\in GL_m(\C)$ and $n\ge 2$. Then $B$ is $n$-cyclotomic if and only if there exists an
invertible matrix $A\in GL_k(\C)$ with $m=kn$, such that $B$ is similar to
\[
M_A=  
\begin{pmatrix}
0 & 0 & 0 & \cdots & 0 & A \cr
I_k & 0 & 0 & \cdots & 0 & 0 \cr
0 & I_k & 0 & \cdots & 0 & 0 \cr
\vdots & \vdots & \vdots & \ddots & \vdots & \vdots \cr
0 & 0 & 0 & \cdots & I_k & 0 
\end{pmatrix}  
\]  
\end{lem}  

\begin{proof}
Let $B$ be $n$-cyclotomic and $\zeta=\zeta_n$. Then
\[
B \sim C_{\zeta}=\diag (C,\zeta C,\ldots ,\zeta^{n-1}C)
\]
for some $C\in GL_k(\C)$ with $m=kn$. Define a matrix $P\in M_m(\C)$ by
\[
P=  
\begin{pmatrix}
C^{n-1} & (\zeta C)^{n-1} & (\zeta^2C)^{n-1} & \cdots & (\zeta^{n-1}C)^{n-1} \cr
\vdots & \vdots & \vdots & \ddots & \vdots \cr
C^2 & (\zeta C)^2 & (\zeta^2C)^2 & \cdots & (\zeta^{n-1}C)^2 \cr 
C & \zeta C & \zeta^2C & \cdots & \zeta^{n-1}C \cr
I_k & I_k & I_k & \cdots & I_k
\end{pmatrix}  
\]  
Let $A=C^n$. A direct computation shows that
\[
M_AP=PC_{\zeta}.
\]
Indeed, both sides are equal to
\[
\begin{pmatrix}
C^n & C^n & \cdots & C^n & C^n \cr
C^{n-1} & (\zeta C)^{n-1} & \cdots & (\zeta^{n-2}C)^{n-1} & (\zeta^{n-1}C)^{n-1} \cr
\vdots & \vdots & \ddots & \vdots & \vdots \cr
C^2 & (\zeta C)^2 & \cdots & (\zeta^{n-2}C)^2  & (\zeta^{n-1}C)^2 \cr 
C & \zeta C & \cdots  & \zeta^{n-2}C & \zeta^{n-1}C 
\end{pmatrix}  
\]
It remains to show that $P$ is invertible. Then $B\sim C_{\zeta}\sim M_A$ as claimed. Note that $P$ is invertible
if and only if $\diag(I_k,C,C^2,\ldots , C^{n-1}) P $ is invertible. But this is true, because
it is the Kronecker product of two invertible matrices, namely the matrix $\diag(I_k,C,C^2,\ldots , C^{n-1}) P $ equals
\[
\begin{pmatrix}
C^{n-1} & \zeta^{n-1}C^{n-1} & \cdots & (\zeta^{n-1})^{n-1}C^{n-1} \cr
C^{n-1} & \zeta^{n-2}C^{n-1} & \cdots & (\zeta^{n-1})^{n-2}C^{n-1} \cr
\vdots & \vdots & \ddots & \vdots \cr
C^{n-1} & \zeta C^{n-1} & \cdots & \zeta^{n-1}C^{n-1} \cr 
C^{n-1} & C^{n-1} & \cdots & C^{n-1}\cr
\end{pmatrix}  
=
\begin{pmatrix}
1 & \zeta^{n-1} & \cdots & (\zeta^{n-1})^{n-1} \cr
1 & \zeta^{n-2} & \cdots & (\zeta^{n-1})^{n-2} \cr
\vdots & \vdots & \ddots & \vdots \cr
1 & \zeta & \cdots & \zeta^{n-1} \cr
1 & 1 & \cdots & 1 
\end{pmatrix}   
\otimes C^{n-1}
\]  
Here the first matrix on the right side is a Vandermonde matrix. \\[0.2cm]
Conversely assume that $B\sim M_A$ for some $A\in GL_k(\C)$. Then there exists a matrix $C\in GL_k(\C)$
such that $A=C^n$. Then we obtain $B\sim M_A\sim C_{\zeta}$ as before in the first part,
so that $B$ is $n$-cyclotomic.
\end{proof}  

\begin{thm}\label{5.4}
Let $\Lg=\La\rtimes \langle v\rangle$ be a complex almost abelian Lie algebra with $Z(\Lg)=0$. Then
$\Lg$ admits an f.p.f.\ automorphism if and only if $\ad(v)_{\mid \La}$ is $n$-cyclotomic for some
$n\ge 2$. In this case $\Lg$ admits an f.p.f.\ automorphism of order $kn$ for all $k\ge 2$.
\end{thm}  

\begin{proof}
Since we have $Z(\Lg)=0$, the linear map $\ad(v)_{\mid \La}: \La\ra \La$ has no eigenvalue zero. Hence it is a
vector space isomorphism. Assume first that $\Lg$ admits an f.p.f.\ automorphism $\phi$. We may assume that $\phi$
is semisimple. Because of $\phi(\La)\subseteq \La$ we have $\phi(v)=\al v+a_v$ with $a_v\in \La$.
It follows that $\al$ is an eigenvalue of $\phi$, so that $\al\neq 0,1$. Let
\[
\La=V_{\la_1}\oplus V_{\la_2}\oplus \cdots \oplus V_{\la_n}
\]
be the decomposition of $\La$ into a direct sum of eigenspaces $V_{\la_i}$ for $\phi$, with distinct eigenvalues
$\la_i$ of $\phi$. Fix an eigenvalue $\la_i$. Then we have $[v,V_{\la_i}]\subseteq V_{\al\la_i}$. Indeed, with
$w\in V_{\la_i}$ we have
\begin{align*}
\phi([v,w]) & =[\phi(v),\phi(w)] \\
            & = [\al v +a_v, \la_i w] \\
            & = \al\la_i[v,w],
\end{align*}
because $[a_v,w]=0$. So we have $[v,w]\in V_{\al\la_i}$ as required. As $\ad(v)_{\mid \La}$ is a vector space
isomorphism of $\La$ we have that
\[
\dim ([v,V_{\la_i}])=\dim (V_{\la_i})\le \dim (V_{\al\la_i})
\]
Hence we obtain a series of subspaces
\[
V_{\la_i}, V_{\al\la_i}, V_{\al^2\la_i},\ldots , V_{\al^k\la_i}, \ldots
\]
of $\La$ of non-decreasing dimension. Since $\La$ is the direct sum of all eigenspaces, these
subspaces cannot be all different. So there exist $p>q$ such that $\al^q\la_i=\al^p\la_i$, or
$\al^{p-q}=1$. It follows that $\al$ is an $n$-th primitive root of unity, where $n\ge 2$ because of $\al\neq 1$.
We have
\[
\dim (V_{\la_1})\le \dim (V_{\al\la_1})\le \cdots \le \dim (V_{\al^n\la_1})=\dim (V_{\la_1}),
\]  
so that we have equality everywhere, and
\[
[v,V_{\al^k\la_1}]=V_{\al^{k+1}\la_1}
\]
for all $k\ge 0$. Let $V_1=V_{\la_1}\oplus V_{\al\la_1}\oplus \cdots \oplus V_{\al^{n-1}\la_1}$. If $V_1\neq \La$,
then there exists an eigenvalue $\la_2$ of $\phi$ such that $\la_2\neq \al^k\la_1$ for all $k$. Then
$V_2=V_{\la_2}\oplus V_{\al\la_2}\oplus \cdots \oplus V_{\al^{n-1}\la_2}$ is another subspace of $\La$ with
$V_1\cap V_2=0$, $V_1\oplus V_2\subseteq \La$ and $[v,V_{\al^k\la_2}]\subseteq V_{\al^{k+1}\la_2}$ for all $k$.
Continuing this way, we find finitely many eigenvalues $\la_1,\ldots , \la_m$, each giving rise to a sequence of
subspaces $V_{\la_i}, V_{\al\la_i},\ldots ,V_{\al^{n-1}\la_i}$ with $[v,V_{\al^k\la_i}]\subseteq V_{\al^{k+1}\la_i}$ for all $k$.
Now let
\begin{align*}
U_1 & = V_{\la_1}\oplus V_{\la_2}\oplus \cdots \oplus V_{\la_m} \\
U_2 & = V_{\al\la_1}\oplus V_{\al\la_2}\oplus \cdots \oplus V_{\al\la_m} \\
\vdots \hspace*{0.18cm} & = \hspace*{0.2cm} \vdots \\
U_n & = V_{\al^{n-1}\la_1}\oplus V_{\al^{n-1}\la_2}\oplus \cdots \oplus V_{\al^{n-1}\la_m} \\
\end{align*}
Note that all spaces $U_i$ have equal dimension, say $\dim (U_i)=k$. Then we have
\begin{align*}
U_1\oplus U_2\oplus \cdots \oplus U_n & = \La, \\
[v,U_i] & = U_{i+1},\; \text{for all} \; i<n,\\
[v,U_n] & = U_1.
\end{align*}  
Choose a basis $\{u_{11},\ldots ,u_{1k}\}$ of $U_1$, and let $u_{2i}=[v,u_{1i}]$ for $i=1,\ldots ,k$.
Then $\{u_{21},\ldots ,u_{2k}\}$ is a basis of $U_2$. Continuing this way, with $u_{j+1,i}=[v,u_{ji}]$ for $j<n$
we obtain a basis of each subspace $U_j$ for $1 \le j\le n$. With respect to the basis
\[
\{u_{11},\ldots ,u_{1k},u_{21},\ldots ,u_{2k}, u_{31},\ldots ,u_{nk}\}
\]  
of $\La$ the matrix of $\ad(v)_{\mid \La}$ is given by
\[
M_A=  
\begin{pmatrix}
0 & 0 & 0 & \cdots & 0 & A \cr
I_k & 0 & 0 & \cdots & 0 & 0 \cr
0 & I_k & 0 & \cdots & 0 & 0 \cr
\vdots & \vdots & \vdots & \ddots & \vdots & \vdots \cr
0 & 0 & 0 & \cdots & I_k & 0 
\end{pmatrix}  
\]
for some $A\in GL_k(\C)$. By Lemma $\ref{5.3}$ it follows that $\ad(v)_{\mid \La}$ is $n$-cyclotomic. \\[0.2cm]
Conversely, assume that  $\ad(v)_{\mid \La}$ is $n$-cyclotomic. Then there is a basis $\{u_{11},u_{12}, \ldots ,u_{nk}\}$
of $\La$ such that the matrix of $\ad(v)_{\mid \La}$ with respect to this basis is given by $M_A$ as above. Now
choose a $\la\in \C^{\times}$ and a primitive $n$-th root of unity $\zeta$ such that
\[
\zeta^{\ell}\la\neq 1 \text{ for all } 0\le \ell \le n-1.
\]
According to the form of $M_A$, we can decompose $\La$ into subspaces $U_i$ such that
\begin{align*}
U_1\oplus U_2\oplus \cdots \oplus U_m & = \La, \\
[v,U_i] & = U_{i+1},\; \text{for all} \; i<n,\\
[v,U_n] & = U_1.
\end{align*}
We define a map $\phi\colon \Lg\ra \Lg$ for $\Lg=\La\rtimes \langle v\rangle$ by $\phi(v)=\zeta v$, and
$\phi_{\mid U_i}$ as the multiplication with $\zeta^{i-1}\la$. So with respect to the above basis, $\phi_{\mid \La}$
is represented by the block matrix
\[
{\rm diag}(\la I_k,\zeta \la I_k,\ldots ,\zeta^{n-1}\la I_k).
\]
Since $\zeta^{\ell}\la\neq 1$, it follows that $\phi$ is fixed-point-free. We claim that
$\phi\in \Aut(\Lg)$. For this it is enough to show that for
each $1\le i\le n$ and for each $u\in U_i$
we have
\[
\phi([v,u])=[\phi(v),\phi(u)].
\]
Since $[v,u]\in U_{i+1}$, with $U_{n+1}=U_1$, we have $\phi([v,u])=\zeta^i\la [v,u]$, and
\[
[\phi(v),\phi(u)]=[\zeta v,\zeta^{i-1}\la u]=\zeta^i\la [v,u],
\]
which was to be shown. \\[0.2cm]
Finally, if $\ad(v)_{\mid \La}$ is $n$-cyclotomic for some $n\ge 2$, we can choose 
$\la=\zeta_{kn}$ in the definition of $\phi$ for every $k\ge 2$. Then
$\phi$ is an f.p.f.\ automorphism of order $kn$. 
\end{proof}  

As an example, we consider the almost abelian unimodular Lie algebra $\Lg=\Lg_9(\al,-1-\al)$ 
for any $\al\neq 0, -1$, with Lie brackets
\[
[e_1,e_2]=e_2,\, [e_1,e_3]=e_2+\al e_3,\, [e_1,e_4]=e_3-(1+\al) e_4.
\]
Then we have $\Lg=\La\rtimes \langle e_1\rangle$, with $\La=\langle e_2,e_3,e_4\rangle$ and $Z(\Lg)=0$.
Furthermore,
\[
A=\ad (e_1)_{\mid \La}=\begin{pmatrix} 1 & 1 & 0 \cr 0 & \al & 1 \cr 0 & 0 & -1-\al \end{pmatrix}.
\]  

\begin{ex}
The almost abelian Lie algebra $\Lg=\Lg_9(\al,-1-\al)$, for $\al\neq 0,-1$, admits an f.p.f.\ automorphism if and only if
$1+\al+\al^2=0$, i.e., if $\al$ is a primitive $3$-rd root of unity.  
\end{ex}  

Indeed, by Theorem $\ref{5.4}$, $\Lg$ admits an f.p.f.\ automorphism if and only if $A$ is $3$-cyclotomic, i.e.,
if and only $A$ is similar to $\diag(c,\om c,\om^2 c)$, for some $c\neq 0$ and $\om=\zeta_3$. The comparison of
the characteristic polynomials, 
\[
t^3-t(\al^2+\al+1)+\al(\al+1)= t^3-c^3,
\]
yields $\al^2+\al+1=0$. Then the matrices are also similar. This coincides with the result in Proposition $\ref{4.3}$.
Note that the algebras $\Lg_7(\al)$ are also almost-abelian, and that we also can recover the result in Proposition $\ref{4.3}$
from Theorem $\ref{5.4}$. \\[0.2cm]
For the general case of an almost-abelian Lie algebra, we will need the following technical lemma.

\begin{lem}\label{extend-to-Jordan-block}
Let $\Lg= \La \rtimes \langle v \rangle$ be an almost abelian Lie algebra. Suppose that $\La = V \oplus W$, where $V$ and $W$
are ideals of $\Lg$, i.e.\,  $V$ and $W$ are invariant under the action of $\ad(v)$, and that there exists a basis 
$e_1, e_2, \ldots , e_n$ of $V$ such that $\ad(v)_{|V}$ is given by means of a Jordan block with zero's on the diagonal:
\[
\ad(v)_{|V} = \begin{pmatrix}
0 & 0 & 0 & \cdots & 0 & 0  \\
1 & 0 & 0 & \cdots & 0 & 0 \\
0 & 1 & 0 & \cdots & 0 & 0\\
\vdots &  & \ddots & \ddots & \vdots & \vdots \\
0 & 0 &   & \ddots & 0 & 0 \\
0 & 0 & 0  & \cdots & 1 & 0 \\
\end{pmatrix}
\]
Suppose that $\varphi'$ is an f.p.f.\ automorphism of the subalgebra $\Lg' = W \rtimes \langle v \rangle$ with
$\varphi'(v)= \zeta v$ for some $\zeta \neq 1$ and $\varphi'(W) = W$. Then $\varphi'$ can be extended to an
f.p.f.\ automorphism $\phi$ of $\Lg$ with $\varphi(V)=V$.
\end{lem}

\begin{proof} The fact that $\ad(v)_{|V}$ is a Jordan block as above means that 
\[
[v,e_i] = e_{i+1} \;\; (1\leq i \leq n-1) \mbox{ and } [v,e_n]=0.
\]
Choose $0\neq \mu\in \C$ in such a way that $\zeta^k\mu\neq 1$ for all $0 \leq k \leq n-1$. Extend $\varphi'$ to a linear map
$\varphi$ on $\Lg$ by defining $\varphi(e_i) = \zeta^{i-1} \mu e_i$. It is obvious that $\varphi$ is an f.p.f.\ invertible
linear map of $\Lg$ with $\varphi(V) =V$. We also have that, for $1\leq i \leq n-1$,
\[
\varphi ([v, e_i]) = \varphi (e_{i+1}) = \zeta^i \mu e_{i+1} = [ \zeta v , \zeta^{i-1} \mu e_i] = [ \varphi(v), \varphi( e_i)]
\]
and 
\[
\varphi([v,e_n]) = \varphi(0) = 0 = [\zeta v, \zeta^{n-1} \mu e _n] = [ \varphi(v) , \varphi(e_n)].
\]
Moreover, for all $w\in W$ and $1\leq i \leq n$ we also have that $\varphi[e_i, w] = 0 = [\varphi(e_i) , \varphi(w) ] $,
showing that $\varphi$ is an automorphism of $\Lg$. 
\end{proof}
Let $\Lg$ be a finite dimensional Lie algebra. The terms $Z_i(\Lg)$ of the upper central series of $\Lg$ are
determined  by  $Z_1(\Lg) = Z(\Lg)$ and $Z_{i+1}(\Lg)/Z_i(\Lg) = Z( \Lg/ Z_i(\Lg)) $ for all $i\geq 1$.
We let $Z_\infty(\Lg) = \bigcup_i Z_i(\Lg)$. Note that $Z_\infty(\Lg) = Z_i(\Lg)$ for some $i$ big enough, since we assume
that $\Lg$ is finite dimensional and so  $Z(\Lg/Z_\infty(\Lg)) =0$. \\[0.2cm]
We are now ready to treat the general case of
almost abelian Lie algebras by reducing it to centerless almost abelian Lie algebras.

\begin{thm} Let $\Lg=\La \rtimes \langle v \rangle$ be an almost abelian Lie algebra. Then $\Lg$ admits an f.p.f.\ automorphism
if and only if $\Lg/Z_\infty(\Lg)$ admits an f.p.f.\ automorphism. 
\end{thm}

For the theorem above, when $\Lg/Z_\infty(\Lg)=0$, we say that  $\Lg/Z_\infty(\Lg)=0$ admits an f.p.f.\ automorphism (indeed, the
only fixed point of the only automorphism is $0$). 

\begin{proof}
Assume that $\Lg$ admits an f.p.f.\ automorphism $\varphi$. As $\varphi(Z_\infty(\Lg)) = Z_\infty(\Lg)$, it follows that
$\varphi$ induces an automorphism $\bar\varphi$ on $\Lg/Z_\infty(\Lg)$ which is clearly also f.p.f.\, i.e., it does not have eigenvalue
equal to $1$. \\[0.2cm]
Conversely, assume now that $\Lg/Z_\infty(\Lg)$ admits an f.p.f.\ automorphism. We choose a basis of $\La$ with respect
to which $\ad(v)_{|\La}$ is represented by a matrix of the form 
\[ \begin{pmatrix} 
J_{n_1} & 0 & \cdots & 0 & 0 \\
0 & J_{n_2} & \cdots & 0 & 0 \\
\vdots & & \ddots & & \vdots \\
0 & 0 & \cdots & J_{n_k} & 0 \\
0 & 0 & \cdots & 0 & B
\end{pmatrix}\]
where each $J_{n_i}$ is a Jordan block of size $n_i$ with zero's on the diagonal and $B$ is a matrix which does
not contain $0$ as an eigenvalue. It is possible that there is no $B$ and we only have Jordan blocks with zero's on
the diagonal. This is the case when $\Lg$ is nilpotent, and so $Z_\infty(\Lg)=\Lg$. \\[0.2cm]
Suppose first that $\Lg$ is not nilpotent (so there is a matrix $B$). Then $Z_\infty(\Lg)$ is exactly the subspace of $\La$
corresponding to the first $n_1 + n_2 + \cdots + n_k$ basis vectors (so the generalised eigenspace of $\ad(v)_{|\La}$ corresponding
to eigenvalue $0$). If we let $W$ denote the space generated by the other generalised eigenvectors, then $\ad(v)_{| W} $ is given
by the matrix $B$. It follows that $\Lg/Z_\infty(\Lg)$ is then isomorphic to the subalgebra
$W \rtimes \langle v \rangle $ of $\Lg$. 
As we assume that $\Lg/Z_\infty(\Lg)$ admits an f.p.f.\ automorphism, by Theorem $\ref{5.4}$ we have that $B$ is $n$-cyclotomic.
From the proof of that theorem we know that there exists an f.p.f.\ automorphism $\varphi'$ of
$W \rtimes \langle v \rangle$ with $\varphi'(W)=W$ and $\varphi(v)=\zeta v$, where $\zeta$ is a primitive $n$-th root if unity. 
By now iteratively applying Lemma~\ref{extend-to-Jordan-block} for each Jordan block $J_{n_i}$, we can extend $\varphi'$ to
an f.p.f.\ automorphism of $\Lg$. \\[0.2cm]
When $\Lg$ is nilpotent (this is the case where $\Lg/Z_\infty(\Lg)=0$), we can start by taking $\varphi'(z) =\zeta z$ for
any choice of $\zeta \in \C\setminus\{0,1\}$ and then also extend $\varphi'$ iteratively to the whole of $\Lg$
(where in the first step we take $W=0$).
\end{proof}

\section{Automorphisms of filiform Lie algebras}

Let $\Lg$ be a complex filiform Lie algebra of dimension $n\ge 3$. It has an adapted basis, which is a basis 
$\{e_1,e_2, \ldots ,e_n\}$ such that
\[
\begin{array}{l}
{[e_1,e_i]}= e_{i+1} \mbox{ for } i= 2,3,\ldots , n-1\\[1mm]
{[e_1,e_n]}= 0 \\[1mm]
{[e_2,e_3]} \in {\rm span}\{ e_5,e_6, \ldots , e_n \}.
\end{array}
\]
For such a basis, it holds that  there is an $\alpha\in \C$ such that
\[
\begin{array}{l}
{[e_i,e_j]}\in \langle e_{i+j}, e_{i+j+1}, \ldots, e_n\rangle \mbox{ if  } i+j \leq n\\[1mm]
{[e_i,e_{n+1-i}]} = (-1)^i \alpha e_n \mbox{ for }i =2,3,\ldots n-1\mbox{ with }\alpha =0 \mbox{ if $n$ is odd.}  
\end{array} 
\]
If $n$ is even, we may assume that $\al=0$ or $\al=1$. For details see section $6$ in \cite{BU82}. Furthermore
the subspaces
\[
\Lg_i= {\rm span} \{ e_i, e_{i+1}, \ldots, e_n\}
\]
are characteristic ideals of $\Lg$. Let us recall the definitions of the filiform Lie algebras $L_n$ and $Q_n$.
\begin{defi}
Let $L_n$ be the standard graded filiform Lie algebra algebra of dimension $n$ with adapted basis $\{e_1,\ldots ,e_n\}$, with
defining Lie brackets $[e_1,e_i]=e_{i+1}$ for $2\le i\le n-1$. 
For $n$ even, let $Q_n$ the graded filiform Lie algebra with adapted basis $\{e_1,\ldots ,e_n\}$, and nonzero Lie brackets defined by
\begin{align*}
[e_1,e_i] & = e_{i+1} \text{ for } 2\le i\le n-1, \\[0.1cm]
[e_i, e_{n+1-i}]  & =(-1)^i e_n \text{ for } 2\le i\le \frac{n}{2}.
\end{align*}
\end{defi}
To each filiform Lie algebra $\Lg$ one can associate a graded Lie algebra by
\[
{\rm gr}(\Lg)=\bigoplus_{i=1}^n \Lg^i/\Lg^{i+1},
\]
where
\[
[x+\Lg^{i+1},y+\Lg^{j+1}]=[x,y]+\Lg^{i+j+1}
\]  
for $x\in \Lg^i$, $y\in \Lg^j$. 
It follows that the associated graded algebra ${\rm gr}(\Lg)$ is either isomorphic to $L_n$ or $Q_n$.

\begin{lem}\label{6.2}
Let $\Lg$ be a filiform Lie algebra satisfying ${\rm gr}(\Lg)\cong Q_n$ and $[\Lg_2,\Lg_2]\subseteq \Lg_n$.
Then $\Lg$ is isomorphic to $Q_n$.  
\end{lem}

\begin{proof}
By assumption $\Lg$ has an adapted basis $\{e_1,\ldots ,e_n\}$, such that
\begin{align*}
[e_1,e_i] & = e_{i+1},\; 2\le i\le n-1,\\
[e_i,e_j] & = c_{i,j}e_n,\; 2\le i,j \le n
\end{align*}
for some $c_{i,j}\in \C$ with $c_{i,n+1-i}=(-1)^i$. We have $c_{i+1,j-1}=-c_{i,j}$ for all $i\ge 2,j\ge 3$ because of
\begin{align*}
[e_{i+1},e_{j-1}] & = [[e_1,e_i],e_{j-1}] \\
                & = -[[e_i,e_{j-1}],e_1] -[[e_{j-1},e_1],e_i] \\
                & = - [e_i,e_j]
\end{align*}
Here we have used that $[e_i,e_{j-1}]\in [\Lg_2,\Lg_2]\subseteq \Lg_n$, and $[e_n,e_1]=0$. The above identity implies
$c_{i,j}=0$ for $i+j \equiv 0\bmod 2$ and $c_{i,j}=(-1)^ic_{2,j+i-2}$ for all $i\ge 2, j\ge 3$. \\[0.2cm]
Using these conditions it is straightforward to show that there are $\be_4,\be_6,\ldots, \be_{n-2}\in \C$ 
such that the linear map $D\colon \Lg\ra \Lg$ defined by
\begin{align*}
D(e_1) & = e_2+\be_4 e_4+\be_6e_6+\cdots +\be_{n-2}e_{n-2},\\
D(e_i) & = e_i, \; 2\le i\le n-1,\\
D(e_n) & = 2e_n 
\end{align*}
is a derivation of $\Lg$. Indeed, one can verify that we have
\[
D([e_i,e_j])=[D(e_i),e_j]+[e_i,D(e_j)]
\]
for all $1\le i<j\le n$, where the $\be_4,\be_6,\ldots, \be_{n-2}\in \C$ are determined by the following system of
equations:
\begin{align*}
0 & =c_{2,3} +\be_4c_{4,3}+\be_6 c_{6,3}+\cdots +\be_{n-6}c_{n-6,3}+ \be_{n-4}c_{n-4,3}+\be_{n-2},\\
0 & =c_{2,5} +\be_4c_{4,5}+\be_6 c_{6,5}+\cdots +\be_{n-6}c_{n-6,5}+\be_{n-4},\\
 &  \hspace*{0.2cm} \vdots \\
0 & =c_{2,n-5}+\be_4c_{4,n-5}+\be_6,\\
0 & = c_{2,n-3}+\be_4.  
\end{align*}
It follows that the rank of $\Lg$ must be equal to $2$. Indeed, the rank of a filiform Lie algebra can only be
$0$, $1$ or $2$,  and $D$ has one eigenvalue $1$, one eigenvalue $2$ and $n-2$ eigenvalues $0$.
Hence the rank cannot be zero or one. For details see \cite{GVE}, Proposition $1$.
This implies that $\Lg\cong Q_n$. 
\end{proof}  

Let us recall the following definition.

\begin{defi}
A Lie algebra $\Lg$ is called {\em characteristically nilpotent}, or a {\em CNLA}, if all derivations
$D\in \Der(\Lg)$ are nilpotent.
\end{defi}

In particular, a characteristically nilpotent Lie algebra $\Lg$ has only nilpotent inner derivations
$D=\ad(x)$ for all $x\in \Lg$, and hence is nilpotent by Engel's theorem.

\begin{prop}\label{6.4}
Let $\Lg$ be a complex filiform Lie algebra of dimension $n\ge 3$. Assume that $\Lg$ is a CLNA. Then $\Lg$ does
not admit an f.p.f.\ automorphism.
\end{prop}  

\begin{proof}
Let $\phi$ be an automorphism of $\Lg$. We need to show that $1$ is an eigenvalue of $\phi$. We may again assume
that $\phi$ is semisimple. Because of $\phi(\Lg_i)=\Lg_i$ for all $i$, the matrix of $\phi$ with respect to
an adapted basis of $\Lg$ is lower-triangular, with eigenvalues $\la_i$, $1\le i\le n$ on the diagonal.
We will show that every $\la_i$ is a power of $\la_1$. For this, we can pass to an adapted basis
$\{e_1',\ldots ,e_n' \}$ with $\phi(e_2')=\la_2e_2'$ as follows: there exists an eigenvector $e_2'=e_2+x_2$ of
$\phi$ corresponding to $\la_2$, where $x_2\in \Lg_3={\rm span}\{e_3,\ldots ,e_n\}$. Define $e_1'=e_1$ and
$e_{i+1}'=[e_1,e_i']$ inductively for all $i\ge 2$. Then we can write $e_i'=e_i+x_i$ with $x_i\in \Lg_{i+1}$ and $x_n=0$.
We have
\[
[e_2',e_3']=[e_2,e_3]+[e_2,x_3]+[x_2,e_3]+[x_2,x_3]\in \Lg_5,
\]
so that the basis is adapted. For convenience, let us rewrite this basis again as $\{e_1,\ldots ,e_n \}$
with $\phi(e_2)=\la_2e_2$. We know that either ${\rm gr}(\Lg)\cong L_n$, or ${\rm gr}(\Lg)\cong Q_n$. \\[0.2cm]
{\em Case 1:} We have ${\rm gr}(\Lg)\cong L_n$. We have $\phi(e_3)=\la_3e_3+y_3$ with $y_3\in \Lg_4$, and also
\begin{align*}
\phi(e_3) & = \phi([e_1,e_2])=[\phi(e_1),\phi(e_2)] \\
          & = [\la_1e_1+y_1,\la_2e_2] \\
          & = \la_1\la_2e_3+\la_2[y_1,e_2]
\end{align*}
with $y_1\in \Lg_2$ and $[y_1,e_2]\in \Lg_4$. This implies $\la_3=\la_1\la_2$. By induction it follows that
\[
\la_i=\la_1^{i-2}\la_2
\]
for all $i\ge 3$. Since $[\Lg_i,\Lg_j]\subseteq \Lg_{i+j}$ for the case ${\rm gr}(\Lg)\cong L_n$, we have
$[e_2,e_{n-1}]=0$. On the other hand, it is impossible that $[e_2,e_k]=0$ for all $k\ge 3$, because
otherwise one inductively obtains $[e_i,e_j]=0$ for all $i,j\ge 2$, so that $\Lg\cong L_n$. This is a contradiction, since
$L_n$ is not a CNLA. Let $j$ be the largest index such that
\[
[e_2,\Lg_2]\subset \Lg_j.
\]
Then we have $[e_2,e_i]=\al e_j+x$ for some $i\ge 3$, with $\al\neq 0$, $x\in \Lg_{j+1}$ and $[e_2,e_k]\in \Lg_j$
for all $k\ge 3$. We also can fix the largest $i$ such that $[e_2,e_i]=\al e_j+x$ with $\al \neq 0$ and $x\in \Lg_{j+1}$.
It follows that $j-i\ge 2$, because we have $[\Lg_2,\Lg_i]\subseteq \Lg_{i+2}$. Applying $\phi$ we obtain
\begin{align*}
[\phi(e_2),\phi(e_i)] & = \phi([e_2,e_i]) \\
        & = \al\phi(e_j)+\phi(x) \\  
        & = \al\la_je_j+z,
\end{align*}
with $z\in \Lg_{j+1}$. 
On the other hand, we have (using $[e_2,\Lg_{i+1}]\subseteq \Lg_{j+1}$)
\begin{align*}
[\phi(e_2),\phi(e_i)] & = [\la_2 e_2,\la_ie_i+y] \mbox{ (for some $y \in \Lg_{i+1}$)}\\
                     & = \la_2\la_i[e_2,e_i]+ \la_2 [e_2,y] \\
                     & = \al\la_2\la_ie_j+w                         
\end{align*}
with $w\in \Lg_{j+1}$. It follows that $\la_2\la_i=\la_j$. Together with $\la_k=\la_1^{k-2}\la_2$ for all $k\geq 2$
this yields $\la_2^2\la_1^{i-2}=\la_2\la_1^{j-2}$, and hence $\la_2=\la_1^{j-i}$.
So the eigenvalues $\la_1,\ldots ,\la_n$ are given by
\[
(\la_1,\ldots ,\la_n)=(\la_1,\la_1^{j-i},\la_1^{j-i+1},\ldots ,\la_1^{j-i+n-2}).
\]
We claim that not all eigenvalues are different. To show this, assume that they are all different, and let
$\Lh_k$ be the eigenspace corresponding to $\la_1^k$ for $k=1,j-i,j-i+1,\ldots ,j-i+n-2$.
Then we have
\[
\Lg=\Lh_1\oplus \Lh_{j-i}\oplus \Lh_{j-i+1}\oplus \cdots \Lh_{j-i+n-2}.
\]
However, this is a positive grading of $\Lg$. For $x\in \Lh_k$ and $y\in \Lh_{\ell}$ we have
\[
\phi([x,y]) =[\phi(x),\phi(y)]=[\la_1^kx,\la_1^{\ell}y]=\la_1^{k+\ell}[x,y],
\]
so that $[x,y]\in \Lh_{k+\ell}$. It follows that $\Lg$ is a not a CNLA, which is a contradiction. Hence we have
$\la_k=\la_{\ell}$ for some $k<\ell$. \\
Assume first that $k\ge 2$. Then we have $\la_1^{j-i+k-2}=\la_1^{j-i+\ell-2}$,
so that $\la_1^{\ell-k}=1$. Let $r=\ell-k\le n-2$.
It follows that the eigenvalue $\la_1$ is an $r$-th root of unity, and the eigenvalues
$\la_1^{j-i},\la_1^{j-i+1},\ldots ,\la_1^{j-i+n-2}$
form a sequence of $n-1$ consecutive powers of $\la_1$, so that at least one of them must be equal to $1$. \\
Now assume that $k=1$, i.e., $\la_1=\la_{\ell}$ for some $\ell>1$. Then $\la_{\ell}=\la_1^{j-i+\ell-2}=\la_1$ and hence 
$\la_1^{j-i+\ell -3}=1$.
If $\ell >2 $ it follows that 
 $\la_{\ell-1}=1$ in which case we are done. If $\ell=2$, this shows that $\lambda_1^{j-i-1}=1$, but also that the sequence of eigenvalues becomes
 \[ (\la_1, \la_2=\la_1, \la_3= \la_1^2,\ldots , \la_n=\la_1^{n-1})\]
 As $2 \leq j-i\leq n-3$, we have that $\la_{j-i}= \la_1^{j-i-1}=1$ is one of the eigenvalues of $\varphi$.
 This concludes the first case. \\[0.2cm]
{\em Case 2:} We have ${\rm gr}(\Lg)\cong Q_n$.
Then we have $[e_2,e_{n-1}]=e_n$, and $\phi$ induces an automorphism on $\Lg/\Lg_n$ with lower-triangular matrix
having eigenvalues $\la_1,\ldots ,\la_{n-1}$ with respect to the basis $\{ e_1+\Lg_n,\ldots ,e_{n-1}+\Lg_n\}$
of $\Lg/\Lg_n$. Clearly we have that ${\rm gr}(\Lg/\Lg_n)\cong L_{n-1}$, where $n-1$ is odd. On the other hand,
suppose that $[\Lg_2,\Lg_2]\subseteq \Lg_n$. Then $\Lg\cong Q_n$ by Lemma $\ref{6.2}$, which is impossible, since
$Q_n$ is not a CNLA. Hence $[\Lg_2,\Lg_2]\not\subseteq \Lg_n$ and $[(\Lg/\Lg_n)_2,(\Lg/\Lg_n)_2]\neq 0$. This implies
$\Lg/\Lg_n\not\cong L_{n-1}$. By the arguments of case $1$ we obtain
\[
(\la_1,\ldots ,\la_{n-1})=(\la_1,\la_1^k,\la_1^{k+1},\ldots , \la_1^{k+n-3})
\]
for the first $n-1$ eigenvalues of $\phi$, with $2\le k\le n-4$. Applying $\phi$ to $[e_2,e_{n-1}]=e_n$
yields
\begin{align*}
\la_ne_n & =\phi(e_n)=[\phi(e_2),\phi(e_n)] \\
         & = [\la_2e_2,\la_{n-1}e_{n-1}]=\la_2\la_{n-1}e_n,
\end{align*}
so that $\la_n=\la_1^{2k+n-3}$ for the last eigenvalue of $\phi$. Now we can proceed as in case $1$.
It follows that not all eigenvalues are different, because $\Lg$ is not graded. Hence we obtain
$\la_i=\la_j$ for some $1\leq i<j\leq n $.\\
If $j<n$, then the induced automorphism $\ov{\phi}$ on
$\Lg/\Lg_n$ also has two equal eigenvalues. Again we can apply the arguments of case $1$ to see that
$\ov{\phi}$ has an eigenvalue equal to $1$. Hence $\phi$ has an eigenvalue $1$ as well. So we are left with the case
that $\la_i=\la_n$ for some $1\le i<n$ (and $\la_i \neq \la_j$ for $1\leq i < j \leq n-1$). Assume first that $i\ge 2$. Then $\la_1^{k+i-2}=\la_1^{2k+n-3}$, which implies
$\la_1^{k+(n+1-i)-2}=1$, so that $\la_{n+1-i}=1$. Now let $i=1$, i.e., we have $\la_n=\la_1$. We will again show that $\Lg$ admits a positive grading, which is a contradiction.

$e_n$ is one of the eigenvectors of $\phi$ with eigenvalue $\la_1=\la_n$. We can find another eigenvector with eigenvalue $\lambda_1$ which is of the form $e_1+x$ with $x\in \Lg_2$. Now let $\Lh_1=\langle e_1+x \rangle$, $\Lh_{2k + n-3}=\langle e_n\rangle$ and $\Lh_i$ be the eigenspace corresponding to eigenvalue $\lambda_i= \lambda_1^i$ for $k\leq i \leq k+n-3$. One can now check that 
\[ \Lg = \Lh_1 \oplus \Lh_k \oplus \Lh_{k+1} \oplus \cdots\Lh_{k+n-3} \oplus \Lh_{2k +n -3} \]
is a positive grading of $\Lg$ (using the fact that  $[ \Lg,\Lg]\subseteq \Lh_k \oplus \Lh_{k+1} \oplus \cdots\Lh_{k+n-3} \oplus \Lh_{2k +n -3}$ ). But this implies again that $\Lg$ is not a CNLA, which concludes the proof. 

\end{proof}  

Finally we are ready to state our main result for complex filiform Lie algebras.
Recall that a Lie algebra
$\Lg$ is called a {\em derived algebra}, if $\Lg=[\Lh,\Lh]$ for some Lie algebra $\Lh$.

\begin{thm}
Let $\Lg$ be a  complex filiform Lie algebra. Then the following statements are equivalent.
\begin{itemize}
\item[(1)] $\Lg$ admits an f.p.f.\ automorphism.
\item[(2)] $\Lg$ is not a CNLA.
\item[(3)] $\Lg$ admits a nonsingular derivation.
\item[(4)] $\Lg$ is a derived Lie algebra. 
\end{itemize}
\end{thm}  

\begin{proof}
$(1)\Rightarrow (2):$ This follows from Proposition $\ref{6.4}$. \\[0.2cm]
$(2)\Rightarrow (3):$ This follows from Proposition $1$ in \cite{GVE}. \\[0.2cm]
$(3)\Rightarrow (1):$ Let $D\colon \Lg\ra \Lg$ be a nonsingular derivation. Note that this implies that
$\Lg$ is nilpotent, but we will not need this. Let $\la_1,\ldots ,\la_n$ be the eigenvalues of $D$. They are
all non-zero by assumption. For $r\in \Q^{\times}$, also $rD$ is a nonsingular derivation. By replacing $D$
with $rD$ for some $r>0$ small enough, we can assume that $\exp(\la_i)\neq 1$ for all $i$. Then
$\exp(D)$ is an automorphism of $\Lg$ with eigenvalues different from $1$, and hence an f.p.f.\ automorphism
of $\Lg$.  \\[0.2cm]
$(4)\Leftrightarrow (2):$ This follows from Proposition $2.14$ in \cite{SIT}. 
\end{proof}

\section*{Acknowledgments}
Dietrich Burde is supported by the FWF, the Austrian Science Foun\-da\-tion, with the 
grant-DOI 10.55776/PAT3410826. 
Karel Dekimpe is supported by Methusalem grant Meth/21/03, a long term structural funding 
of the Flemish Government, and by a FWO sabbatical bench fee K803525N.
The authors would like to thank the Erwin Schrödinger International Institute for 
Mathematics and Physics (ESI) for its hospitality during their research stays.
For open access purposes, the authors have applied a CC BY public copyright license to 
any author-accepted manuscript version arising from this submission.

\end{document}